\documentclass[a4paper, 12pt]{amsart}
\usepackage{amssymb, tikz, hyperref}
\usepackage[all]{xy}
\usetikzlibrary{decorations.markings}
\usetikzlibrary{shapes, decorations.text}
\usetikzlibrary{arrows.meta}

\newtheorem{thm}{Theorem}[section]
\newtheorem{cor}[thm]{Corollary}
\newtheorem{lem}[thm]{Lemma}
\newtheorem{prp}[thm]{Proposition}

\newtheorem{theoremintro}{Theorem}

\theoremstyle{definition}
\newtheorem{dfn}[thm]{Definition}

\theoremstyle{remark}
\newtheorem{rmk}[thm]{Remark}

\numberwithin{equation}{section}

\newcommand{\KK}{\mathbb{K}}
\newcommand{\NN}{\mathbb{N}}

\newcommand{\TT}{\mathbb{T}}
\newcommand{\ZZ}{\mathbb{Z}}

\newcommand{\Bb}{\mathcal{B}}

\newcommand{\Hh}{\mathcal{H}}

\newcommand{\Kk}{\mathcal{K}}
\newcommand{\Ll}{\mathcal{L}}

\newcommand{\Uu}{\mathcal{U}}

\newcommand{\gr}{\operatorname{gr}}
\newcommand{\id}{\operatorname{id}}

\newcommand{\lt}{\operatorname{lt}}

\mathchardef\hyphen="2D

\makeatletter
\providecommand{\leftsquigarrow}{%
  \mathrel{\mathpalette\reflect@squig\relax}%
}
\newcommand{\reflect@squig}[2]{%
  \reflectbox{$\m@th#1\rightsquigarrow$}%
}
\makeatother

\title{Amplified graph C*-algebras II: reconstruction}
\author[Eilers]{S{\o}ren Eilers}
\email[S. Eilers]{eilers@math.ku.dk}
\address[S. Eilers]{Department of Mathematical Sciences\\
University of Co\-penhagen\\
Universitetsparken 5\\
DK-2100 Copenhagen\\
Denmark}
\author[Ruiz]{Efren Ruiz}
\email[E. Ruiz]{ruize@hawaii.edu}
\address[E. Ruiz]{Department of Mathematics\\
University of Hawaii, Hilo\\
200W. Kawili St.\\
Hilo\\
Hawaii\\
96720-4091 USA}
\author[Sims]{Aidan Sims}
\email[A. Sims]{asims@uow.edu.au}
\address[A. Sims]{School of Mathematics and Applied Statistics\\
The University of Wollongong\\
NSW  2522\\
Australia}

\thanks{This research was supported by Australian Research Council Discovery Project DP200100155,
by DFF-Research Project 2 `Automorphisms and Invariants of Operator Algebras', no. 7014-00145B, and
by a Simons Foundation Collaboration Grant, \#567380}

\begin{document}

\begin{abstract}
Let $E$ be a countable directed graph that is amplified in the sense that whenever there is an
edge from $v$ to $w$, there are infinitely many edges from $v$ to $w$.  We show that $E$ can be
recovered from $C^*(E)$ together with its canonical gauge-action, and also from $L_\KK(E)$
together with its canonical grading.
\end{abstract}

\maketitle

\section{Introduction}

The purpose of this paper is to investigate the gauge-equivariant isomorphism question for
$C^*$-algebras of countable amplified graphs, and the graded isomorphism question for Leavitt path
algebras of countable amplified graphs. A directed graph $E$ is called an \emph{amplified} graph
if for any two vertices $v, w$, the set of edges from $v$ to $w$ is either empty or infinite.

The geometric classification (that is, classification by the underlying graph modulo the
equivalence relation generated by a list of allowable graph moves) of the $C^*$-algebras of
finite-vertex amplified graph $C^*$-algebras was completed in \cite{ERS-amplified}, and was an
important precursor to the eventual geometric classification of all finite graph $C^*$-algebras
\cite{ERRS-classification}. But there has been increasing recent interest in understanding
isomorphisms of graph $C^*$-algebras that preserve additional structure: for example the canonical
gauge action of the circle; or the canonical diagonal subalgebra isomorphic to the algebra of
continuous functions vanishing at infinity on the infinite path space of the graph; or the smaller
coefficient algebra generated by the vertex projections; or some combination of these (see, for
example, \cite{Brix, BCW, BLRS, CEOR, CRS, DEG, MM}).

A program of geometric classification for these various notions of isomorphism was initiated by
the first two authors in \cite{ER}. They discuss $\mathsf{xyz}$-isomorphism of graph
$C^*$-algebras, where $\mathsf{x}$ is 1 if we require exact isomorphism, and 0 if we require only
stable isomorphism; $\mathsf{y}$ is 1 if the isomorphism is required to be gauge-equivariant, and
0 otherwise; and $\mathsf{z}$ is 1 if the isomorphism is required to preserve the diagonal
subalgebra and 0 otherwise. They also identified a set of moves on graphs that preserve various
kinds of $\mathsf{xyz}$-isomorphism, and conjectured that for all $\mathsf{xyz}$ other than
$\mathsf{x10}$, the equivalence relation on graphs with finitely many vertices induced by
$\mathsf{xyz}$-isomorphism of $C^*$-algebras is generated by precisely those of their moves that
induce $\mathsf{xyz}$-isomorphisms.

This was an important motivation for the present paper. None of the moves in \cite{ER} takes an
amplified graph to an amplified graph. And although we know of one important instance where one
amplified graph can be transformed into another via a sequence of $\mathsf{101}$-preserving moves
passing through non-amplified graphs (see Diagram~\eqref{eq:ORRO} in Remark~\ref{rm:ER
conjecture}), we had given up on envisioning such a sequence consisting only of
$\mathsf{x1z}$-preserving moves. Based on the main conjecture of \cite{ER}, this led us to expect
that an amplified graph $C^*$-algebra together with its gauge action should remember the graph
itself.

Our main theorem shows that, indeed, any countable amplified graph $E$ can be reconstructed from
either the circle-equivariant $K_0$-group of its $C^*$-algbra, or the graded $K_0$-group of its
Leavitt path algebra over any field. That is:

\begin{theoremintro}\label{thm-main}
Let $E$ and $F$ be countable amplified graphs and let $\KK$ be a field.  Then the following are
equivalent:
\begin{enumerate}
\item\label{it:main graph iso} $E \cong F$;
\item\label{it:main alg K-th} there is a $\ZZ[x, x^{-1}]$-module order-isomorphism
    $K_0^{\gr}(L_{\KK}(E)) \cong K_0^{\gr}(L_{\KK}(F))$; and
\item\label{it:main C* K-th} there is a $\ZZ[x, x^{-1}]$-module order-isomorphism
    $K_0^{\TT}(C^*(E), \gamma)) \cong K_0^{\TT}(C^*(F), \gamma))$ of $\TT$-equivariant
    $K_0$-groups.
\end{enumerate}
\end{theoremintro}

We spell out a number of consequences of this theorem in Remark~\ref{rmk:field indep},
Theorem~\ref{thm:C*-consequences}, and Theorem~\ref{thm:LPA-consequences}. The headline is that
for amplified graphs, and for any $\mathsf{x, z}$, the graph $C^*$-algebras $C^*(E)$ and $C^*(F)$
are $\mathsf{x1z}$-isomorphic if and only if $E$ and $F$ are isomorphic. Combined with results of
\cite{AER-geometric, ERRS-classification}, this confirms \cite[Conjecture~5.1]{ER} for amplified
graphs (see Remark~\ref{rm:ER conjecture}).

Another immediate consequence is that, since ordered graded $K_0$ is an isomorphism invariant of
graded rings, and ordered $\TT$-equivariant $K_0$ is an isomorphism invariant of $C^*$-algebras
carrying circle actions, our theorem confirms a special case of Hazrat's conjecture: ordered
graded $K_0$ is a complete graded-isomorphism invariant for amplified Leavitt path algebras; and
we also obtain that ordered $\TT$-equivariant $K_0$ is a complete gauge-isomorphism invariant of
amplified graph $C^*$-algebras.

A third consequence is related to different graded stabilisations of Leavitt path algebras (and
different equivariant stabilisations of graph $C^*$-algebras). Each Leavitt path algebra has a
canonical grading, and, as alluded to above, significant work led by Hazrat has been done on
determining when graded K-theory completely classifies graded Leavitt path algebras. Historically,
in the classification program for $C^*$-algebras, significant progress has been made by first
considering classification up to stable isomorphism; so it is natural to consider the same
approach to Hazrat's graded classification question. But almost immediately, there is a
difficulty: which grading on $L_{\KK}(E) \otimes M_\infty(\KK)$ should we consider? It seems
natural enough to use the grading arising from the graded tensor product of the graded algebras
$L_{\KK}(E)$ and $M_\infty(\KK)$. But there are many natural gradings on $M_\infty(\KK)$: given
any $\overline{\delta} \in \prod_i \ZZ$, we obtain a grading of $M_\infty(\KK)$ in which the $m,n$
matrix unit is homogeneous of degree $\overline{\delta}_m - \overline{\delta}_n$. Different
nonzero choices for $\overline{\delta}$ correspond to different ways of stabilising $L_{\KK}(E)$
by modifying the graph $E$ (for example by adding heads \cite{Tomforde-Stab}), while taking
$\overline{\delta} = (0, 0, 0, \dots)$ corresponds to stabilising the associated groupoid by
taking its cartesian product with the (trivially graded) full equivalence relation $\NN \times
\NN$.

In Section~\ref{sec:stabilizations}, we show that for amplified graphs it doesn't matter what
value of $\overline{\delta}$ we pick. Specifically, using results of Hazrat, we prove that
$K_0^{\gr}(L_{\KK}(E) \otimes M_\infty(\KK)(\overline{\delta})) \cong K_0^{\gr}(L_\KK(E))$
regardless of $\overline{\delta}$. Consequently, for any choice of $\overline{\delta}$ we have
$L_{\KK}(E) \otimes M_\infty(\KK)(\overline{\delta}) \cong L_{\KK}(F) \otimes
M_\infty(\KK)(\overline{\delta})$ if and only if there exists a $\ZZ[x, x^{-1}]$-module
order-isomorphism $K_0^{\gr}(L_{\KK}(E)) \cong K_0^{\gr}(L_{\KK}(F))$. A similar result holds for
$C^*$-algebras with the gradings on Leavitt path algebras replaced by gauge actions on graph
$C^*$-algebras, and the gradings of $M_\infty(\KK)$ corresponding to different elements
$\overline{\delta}$ replaced by the circle actions on $\Kk(\ell^2)$ implemented by different
strongly continuous unitary representations of the circle on $\ell^2$.

We prove our main theorem in Section~\ref{sec:prof of main thm}. We use general results to see
that the graded $K_0$-group of $L_\KK(E)$ and the equivariant $K_0$-group of $C^*(E)$ are
isomorphic as ordered $\ZZ[x, x^{-1}]$-modules to the $K_0$-groups of the Leavitt path algebra and
the graph $C^*$-algebra (respectively) of the skew-product graph $E \times_1 \ZZ$. These are known
to coincide, and their lattice of order ideals (with canonical $\ZZ$-action) is isomorphic to the
lattice of hereditary subsets of $(E \times_1 \ZZ)^0$ with the $\ZZ$-action of translation in the
second variable. So the bulk of the work in Section~\ref{sec:prof of main thm} goes into showing
how to recover $E$ from this lattice. We then go on in Section~\ref{sec:stabilizations} to
establish the consequences of our main theorem for stabilizations. Here the hard work goes into
showing that $K_0^{\gr}(L_{\KK}(E) \otimes M_\infty(\KK)(\overline{\delta})) \cong
K_0^{\gr}(L_\KK(E))$ for any $\overline{\delta} \in \prod_i \ZZ$ and that $K_0^\TT(C^*(E) \otimes
\Kk, \gamma^E \otimes \operatorname{Ad}_u) \cong K_0^{\gr}(L_\KK(E))$ for any strongly continuous
unitary representation $u$ of $\TT$.

%
%

\section{Gauge-invariant classification of amplified graph \texorpdfstring{$C^*$}{C*}-algebras}\label{sec:prof of main thm}

Throughout the paper, a countable directed graph $E$ is a quadruple $E = (E^0, E^1, r, s)$ where
$E^0$ is a countable set whose elements are called \emph{vertices}, $E^1$ is a countable set whose
elements are called \emph{edges}, and $r, s \colon E^1 \to E^0$ are functions. We think of the
elements of $E^0$ as points or dots, and each element $e$ of $E^1$ as an arrow pointing from the
vertex $s(e)$ to the vertex $r(e)$. We follow the conventions of, for example
\cite{FLR-graphconv}, where a path is a sequence $e_1 \dots e_n$ of edges in which $s(e_{n+1}) =
r(e_n)$. This is not the convention used in Raeburn's monograph \cite{Rae-graphconv}, but is the
convention consistent with all of the Leavitt path algebra literature as well as much of the graph
$C^*$-algebra literature. In keeping with this, for $v, w \in E^0$ and $n \ge 0$, we define
\[
v E^1 = s^{-1}(v), \quad E^1 w = r^{-1}(w), \quad \text{and} \quad v E^1 w = s^{-1} (v) \cap r^{-1} (w).
\]
We will also write $vE^n$ for the sets of paths of length $n$ that are emitted by $v$, $E^n w$ for
the set of paths of length $n$ received by $w$, and $vE^n w$ for the set of paths of length $n$
pointing from $v$ to $w$.

A vertex $v$ is \emph{singular} if $vE^1$ is either empty or infinite, so $v$ is either a sink or
an infinite emitter; and for any edge $e$, we have $s_e^*s_e = p_{r(e)}$ and $p_{s(e)} \ge s_e
s^*_e$ in the graph $C^*$-algebra $C^*(E)$.  We will also consider the Leavitt path algebras,
$L_\KK(E)$ for any field $\KK$, the so-called algebraic cousin of graph $C^*$-algebras. Leavitt
path algebras are defined via generators and relations similar to those for graph $C^*$-algebras
(see \cite{AP}).

Countable directed graphs $E$ and $F$ are \emph{isomorphic}, denoted $E \cong F$, if there is a
bijection $\phi : E^0 \sqcup E^1 \to F^0 \sqcup F^1$ that restricts to bijections $\phi^0 \colon
E^0 \to F^0$ and $\phi^1 \colon E^1 \to F^1$ such that
\[
\phi^0 ( r (e) ) = r ( \phi^1 (e)) \quad \text{and} \quad \phi^0 ( s(e) ) = s ( \phi^1 (e)).
\]

In this paper, we consider amplified graphs. The classification of amplified graph $C^*$-algebras
was the starting point in the classification of unital graph $C^*$-algebras via moves (see
\cite{ERS-amplified} and \cite{ERRS-classification}).

\begin{dfn}[Amplified Graph and Amplified graph algebra] A directed graph $E$ is an
\emph{amplified graph} if for all $v, w \in E^0$, the set $v E^1 w = s^{-1}(v) \cap r^{-1} (w)$ is
either empty or infinite.  An \emph{amplified graph $C^*$-algebra} is a graph $C^*$-algebra of an
amplified graph and an \emph{amplified Leavitt path algebra} is a Leavitt path algebra of an
amplified graph.
\end{dfn}

Observe that in an amplified graph, every vertex is singular.

Recall that a set $H \subseteq E^0$ is \emph{hereditary} if $s(e) \in H$ implies $r(e) \in H$ for
every $e \in E^1$, and is \emph{saturated} if whenever $v$ is a regular vertex such that $r(vE^1)
\subseteq H$, we have $v \in H$. Again since every vertex in an amplified graph is singular, every
set of vertices is saturated.

Recall from \cite{KP-groupactions} that if $E$ is a directed graph, then the skew-product graph $E
\times_1 \ZZ$ is the graph with vertices $E^0 \times \ZZ$ and edges $E^1 \times \ZZ$ with $s(e,n)
= (s(e), n)$ and $r(e, n) = (r(e), n+1)$. If $E$ is an amplified graph, then so is $E \times_1
\ZZ$.

For a countable amplified graph, $E$, we write $\Hh(E \times_1 \ZZ)$ for the lattice (under set
inclusion) of hereditary subsets of the vertex-set of the skew-product graph $E \times_1 \ZZ$. The
action of $\ZZ$ on $E \times_1 \ZZ$ by given by $n \cdot (e, m) = (e, n+m)$ induces an action
$\lt$ of $\ZZ$ on $\Hh(E \times_1 \ZZ)$. There is also a distinguished element $H_0 \in \Hh(E
\times_1 \ZZ)$ given by $H_0 := \{(v, n) : v \in E^0, n \ge 0\} \subseteq (E \times_1 \ZZ)^0$.

Throughout this section, given $v \in E^0$ and $n \in \ZZ$, we write $H(v,n)$ for the smallest
hereditary subset of $(E \times_1 \ZZ)^0$ containing $(v,n)$. So $H(v, n) = \{(r(\mu), n + |\mu|)
: \mu \in vE^*\}$ is the set of vertices that can be reached from $(v, n)$ in $E \times_1 \ZZ$.

If $(\Ll, \preceq)$ is a lattice, we say that $L \in \Ll$ has a unique predecessor if there exists
$K \in \Ll$ such that $K \prec L$, and every $K'$ with $K' \prec L$ satisfies $K' \preceq K$. The
next proposition is the engine-room of our main result.

\begin{prp}\label{prp:key}
Let $E$ be a countable amplified graph. Define $\Hh_{\operatorname{vert}} \subseteq \Hh(E \times_1
\ZZ)$ to be the subset
\[
\Hh_{\operatorname{vert}} = \{H \in \Hh(E \times_1 \ZZ) : H \text{ has a unique predecessor}\}.
\]
Then $\Hh_{\operatorname{vert}} = \{H(v,n) : v \in E^0\text{ and }n \in \ZZ\}$. Let
\[
\overline{E}^0 := \{H \in \Hh_{\operatorname{vert}} : H \subseteq H_0\text{ and }H \not\subseteq \lt_1(H_0)\}.
\]
Define $\overline{E}^1 := \{(H, n, K) : H, K \in \overline{E}^0, \lt_1(K) \subseteq H,\text{ and
}n \in \NN\}$. Define $\bar{s}, \bar{r} : \overline{E}^1 \to \overline{E}^0$ by $\bar{s}(H, n, K)
= H$ and $\bar{r}(H, n, K) = K$. Then $\overline{E} := (\overline{E}^0, \overline{E}^1, \bar{r},
\bar{s})$ is a countable amplifed directed graph, and there is an isomorphism $E \cong
\overline{E}$ that carries each $v \in E^0$ to the hereditary subset of $(E \times_1 \ZZ)^0$
generated by $(v,0)$.
\end{prp}

\begin{proof}
The argument of \cite[Lemma~5.2]{ERS-amplified} shows that $\Hh_{\operatorname{vert}} = \{H(v, n)
: v \in E^0, n \in \ZZ\}$.

We clearly have $H(v, n) \subseteq H_0$ if and only if $n \ge 0$, and $H(v, n) \subseteq
\lt_1(H_0)$ if and only if $n \ge 1$, so $\overline{E}^0 = \{H(v, 0) : v \in E^0\}$. Since $E
\times_1 \ZZ$ is acyclic, the $H(v, 0)$ are distinct, and we deduce that $\theta^0 : v \mapsto
H(v, 0)$ is a bijection from $E^0$ to $\overline{E}^0$.

Fix $v, w \in E^0$. We have $\lt_1(H(w, 0)) = H(w, 1)$, and since $(w, 1) \in H(v, 0)$ if and only
if $v E^1 w \not= \emptyset$, we have $H(w, 1) \subseteq H(v, 0)$ if and only if $v E^1 w \not=
\emptyset$, in which case $v E^1 w$ is infinite because $E$ is amplified. It follows that $|H(v,
0) \overline{E}^1 H(w, 0)| = |v E^1 w|$ for all $v,w$, so we can choose a bijection $\theta^1 :
E^1 \to \overline{E}^1$ that restricts to bijections $v E^1 w \to \theta^0(v) \overline{E}^1
\theta^0(w)$ for all $v,w \in E^0$. The pair $(\theta^0, \theta^1)$ is then the desired
isomorphism $E \cong \overline{E}$.
\end{proof}

In order to use Proposition~\ref{prp:key} to prove Theorem~\ref{thm-main}, we need to know that if
$(\Hh(E \times_1 \ZZ), \lt^E)$ is order isomorphic to $(\Hh(F \times_1 \ZZ), \lt^F)$ then there is
an isomorphism from $(\Hh(E \times_1 \ZZ), \lt^E)$ to $(\Hh(F \times_1 \ZZ), \lt^F)$ that carries
$H^E_0$ to $H^F_0$. We do this by showing that if $E$ is connected, then we can recognise the sets
$\lt_n(H_0)$ amongst all the hereditary subsets of $(E \times_1 \ZZ)^0$ using just the
order-structure and the action $\lt$.

Recalling that $v E^n w$ denotes the set of paths of length $n$ from $v$ to $w$, we have
\begin{equation}\label{eq:connections}
    H(w, n) \subseteq H(v, m)\quad\text{ if and only if }\quad v E^{n-m} w \not= \emptyset.
\end{equation}

Recall that a graph $E$ is said to be \emph{connected} if the smallest equivalence relation on
$E^0$ containing $\{(s(e), r(e)) : e \in E^1\}$ is all of $E^0 \times E^0$.

Let $E$ be a connected, countable amplified graph. The set $V_0 := \{H(v,0) : v \in E^0\}$ is
exactly the set of maximal elements of the collection $\{H \in \Hh_{\operatorname{vert}} : H
\subseteq H_0\}$. The sets $H_0$ and $V_0$ have the following properties:
\begin{itemize}
    \item for each $H \in \Hh_{\operatorname{vert}}$ there is a unique $n \in \ZZ$ such that
        $\lt_n(H) \in V_0$;
    \item the smallest equivalence relation on $V_0$ containing $\{(H,K) : \lt_1(K) \subseteq
        H\}$ is all of $V_0 \times V_0$; and
    \item if $H, K$ are distinct elements of $V_0$, and if $n \ge 0$, then $H \not\subseteq
        \lt_n(K)$.
\end{itemize}
The next lemma shows that for connected graphs, these properties characterise $H_0$ up to
translation.

\begin{lem}\label{lem:H automatic}
Suppose that $E$ is a connected, countable amplified graph. Take $H \in \Hh(E \times_1 \ZZ)$, and
let $V_H$ be the set of maximal elements of $\{K \in \Hh_{\operatorname{vert}} : K \subseteq H\}$
with respect to set inclusion. Suppose that
\begin{enumerate}
    \item for each $K \in \Hh_{\operatorname{vert}}$ there is a unique $n \in \ZZ$ such that
        $\lt_n(K) \in V_H$;
    \item the smallest equivalence relation on $V_H$ containing $\{(H,K) : \lt_1(K) \subseteq
        H\}$ is all of $V_H \times V_H$; and
    \item if $K, K'$ are distinct elements of $V_H$, and if $n \ge 0$, then $K \not\subseteq
        \lt_n(K')$.
\end{enumerate}
Then there exists $n \in \ZZ$ such that $H = \lt_n(H_0)$.
\end{lem}
\begin{proof}
For each $v \in E^0$, item~(1) applied to $K = H(v,0)$ shows that there exists a unique $n_v \in
\ZZ$ such that $H(v, n_v) = \lt_{n_v}(K) \in V_H$. So $V_H = \{H(v, n_v) : v \in E^0\}$. We must
show that $n_v = n_w$ for all $v,w \in E^0$. To do this, it suffices to show that for any $u \in
E^0$, we have $n_w \ge n_u$ for all $w \in E^0$.

So fix $u \in E^0$. Define
\[
L_u := \{v \in E^0 : n_v < n_u\}\quad\text{ and }\quad
G_u := \{w \in E^0 : n_w \ge n_u\}
\]
We prove that if $v \in L_u$ and $w \in G_u$, then
\begin{equation}\label{eq:no containment}
\lt_1(H (v, n_v) ) \not\subseteq H(w, n_w)
    \quad\text{ and }\quad
\lt_1(H (w, n_w)) \not\subseteq H (v, n_v).
\end{equation}
For this, fix $v \in L_u$ and $w \in G_u$; note that in particular $v \not= w$.

To see that $\lt_1(H (v, n_v)) \not\subseteq H(w, n_w)$, suppose otherwise for contradiction. Then
$H (v, n_v+1) \subseteq H(w, n_w)$. Hence~\eqref{eq:connections} shows that $w E^{n_v + 1 - n_w}
v \not= \emptyset$, which forces $n_v \ge n_w - 1$. Since $v \in L_u$ and $w \in G_u$, we also
have $n_v \le n_w - 1$, and we conclude that $n_v + 1 - n_w = 0$. This forces $w E^0 v \not=
\emptyset$, contradicting that $v \not= w$.

To see that $\lt_1(H(w, n_w)) \not\subseteq H(v, n_v)$, we first claim that there is no $e \in
E^1$ satisfying $s(e) \in L_u$ and $r(e) \in G_u$. To see this, fix $x \in L_u$ and $y \in G_u$.
Then $n_y > n_x$, and in particular $n_y - 1 - n_x \ge 0$. Hence Item~(3) shows that $H(y, n_y)
\not\subseteq \lt_{n_y - 1 - n_x}(H(x, n_x))$. Applying $\lt_{1-n_y}$ on both sides shows that
$\lt_1(H(y, 0)) \not\subseteq H(x, 0)$, and so $x E^1 y = \emptyset$. This proves the claim.

Since $v \in L_u$, applying the claim $n_w + 1 - n_v$ times shows that for any path $\mu \in
vE^{n_w + 1 - n_v}$, we have $r(\mu) \in L_u$. In particular, $v E^{n_w + 1 - n_v} w = \emptyset$.
Thus~\eqref{eq:connections} implies that $\lt_1(H(w, n_w)) \not\subseteq H(v, n_v)$.

We have now established~\eqref{eq:no containment}. Set
\[
\overline{L}_u = \{ H( v, n_v ) : v \in L_u \} \quad \text{and} \quad \overline{G}_u = \{ H( w, n_w ): w \in G_u \}.
\]
Then~\eqref{eq:no containment} shows that $(\overline{L}_u \times \overline{L}_u) \sqcup
(\overline{G}_u \times \overline{G}_u)$ is an equivalence relation on $V_H$ containing $\{(H,K) :
\lt_1(K) \subseteq H\}$. Thus item~(2) implies that either $\overline{L}_u$ or $\overline{G}_u$ is
empty. Since $H(u, n_u) \in \overline{G}_u$, we deduce that $\overline{L}_u = \emptyset$ which
implies that $L_u = \emptyset$.  Hence $G_u = E^0$, and so $n_w \ge n_u$ for all $w \in E^0$ as
required.
\end{proof}

\begin{cor}\label{cor:H carried}
Suppose that $E$ and $F$ are amplified graphs. If there exists an isomorphism $\rho : (\Hh(E
\times_1 \ZZ), \subseteq, \lt^E) \cong (\Hh(F \times_1 \ZZ), \subseteq, \lt^F)$, then there exists
an isomorphism $\overline{\rho} : (\Hh(E \times_1 \ZZ), \subseteq, \lt^E) \to (\Hh(F \times_1
\ZZ), \subseteq, \lt^F)$ such that $\overline{\rho}(H^E_0) = H^F_0$.
\end{cor}
\begin{proof}
First suppose that $E$ and $F$ are connected as in Lemma~\ref{lem:H automatic}. Since $H \in
\Hh_{\operatorname{vert}}^E$ if and only if $H$ has a unique predecessor in $\Hh(E \times_1 \ZZ)$
and similarly for $F$, the map $\rho$ restricts to an inclusion-preserving bijection $\rho :
\Hh_{\operatorname{vert}}^E \to \Hh_{\operatorname{vert}}^F$. Since $H^E_0$ and $V^E_0$ satisfy
(1)--(3) of Lemma~\ref{lem:H automatic}, so do $\rho(H^E_0)$ and $\{\rho(H) : H \in V^E_0\}$. So
Lemma~\ref{lem:H automatic} shows that $\rho(H^E_0) = \lt_n(H^F_0)$ for some $n \in \ZZ$, and
therefore $\overline{\rho} := \lt_{-n} \circ \rho$ is the desired isomorphism.

Now suppose that $E$ and $F$ are not connected. Let $\mathcal{WC}(E)$ denote the set of
equivalence classes for the equivalence relation on $E^0$ generated by $\{(s(e), r(e)) : e \in
E^1\}$; so the elements of $\mathcal{WC}(E)$ are the weakly connected components of $E$.
Similarly, let $\mathcal{WC}(F)$ be the set of weakly connected components of $F$.

Using that $v E^* w$ is nonemtpy if and only if $\lt_n(H(w,0)) \subseteq H(v,0)$ for some $n \in
\ZZ$, we see that $v E^* w \not= \emptyset$ if and only if $\bigcup_n \lt_n(H(w, i)) \subseteq
\bigcup_n \lt_n(H(v, j))$ for some (equivalently for all) $i,j \in \ZZ$. Since the same is true in
$F$, we see that for $v,w \in E^0$, writing $x,y \in F^0$ for the elements such that $\rho(H(v,
0)) \in \lt_\ZZ(H(x, 0))$ and $\rho(H(w, 0)) \in \lt_\ZZ(H(y,0))$, we have $v E^* w \not=
\emptyset$ if and only if $x F^* y \not= \emptyset$. Now an induction shows that there is a
bijection $\tilde{\rho} : \mathcal{WC}(E) \to \mathcal{WC}(F)$ such that for each $C \in
\mathcal{WC}(E)$, we have $\rho(\{H(v, n) : v \in C, n \in \ZZ\}) = \{H(w, m) : w \in
\tilde{\rho}(C), m \in \ZZ\}$. For each $C \in \mathcal{WC}(E)$, write $E_C$ for the subgraph $(C,
CE^1 C, r, s)$ of $E$ and similarly for $F$. Then the inclusions $E_C \hookrightarrow E$ induce
inclusions $(H(E_C \times_1 \ZZ), \lt) \hookrightarrow (H(E \times_1 \ZZ), \lt)$ whose ranges are
$\lt$-invariant and mutually incomparable with respect to $\subseteq$. Hence $\rho$ induces
isomorphisms $\rho_C : (\Hh(E_C \times_1 \ZZ), \lt) \cong (\Hh(F_{\tilde{\rho}(C)} \times_1 \ZZ),
\lt)$. The first paragraph then shows that for each $C \in \mathcal{WC}(E)$ there is an
isomorphism $\overline{\rho}_C : (\Hh(E_C \times_1 \ZZ), \lt) \to (\Hh(F_{\tilde{\rho}(C)}
\times_1 \ZZ), \lt)$ that carries $H^{E_C}_0$ to $H^{F_{\tilde\rho(C)}}_0$, and these then
assemble into an isomorphsm $\overline{\rho} : (\Hh(E \times_1 \ZZ), \subseteq, \lt^E) \to (\Hh(F
\times_1 \ZZ), \subseteq, \lt^F)$ such that $\overline{\rho}(H^E_0) = H^F_0$.
\end{proof}

We are now ready to prove Theorem~\ref{thm-main}.

\begin{proof}[Proof of Theorem~\ref{thm-main}]
That (\ref{it:main graph iso})~implies~(\ref{it:main alg K-th}) and that (\ref{it:main graph
iso})~implies~(\ref{it:main C* K-th}) are clear.

By \cite[Proposition~5.7]{AHLS} the graded $\mathcal{V}$-monoid
$\mathcal{V}^{\operatorname{gr}}(L_{\KK}(E))$ is isomorphic to the $\mathcal{V}$-monoid
$\mathcal{V}(L_{\KK}(E \times_1 \ZZ))$, and that this isomorphism is equivariant for the canonical
$\ZZ[x, x^{-1}]$ actions arising from the grading on $\mathcal{V}^{\operatorname{gr}}(L_{\KK}(E))$
and from the action on $\mathcal{V}(L_{\KK}(E \times_1 \ZZ))$ induced by translation in the
$\ZZ$-coordinate in $E \times_1 \ZZ$. Hence $K^{\gr}_0(L_{\KK}(E))$ is order isomorphic to
$K_0(L_{\KK}(E \times_1 \ZZ))$ as $\ZZ[x, x^{-1}]$-modules. Hence condition~(\ref{it:main alg
K-th}) holds if and only if $K_0(L_{\KK}(E \times_1 \ZZ)) \cong K_0(L_{\KK}(F \times_1 \ZZ))$ as
ordered $\ZZ[x, x^{-1}]$-modules.

Likewise \cite[Theorem~2.7.9]{NCP-LectureNotes} shows that the equivariant $K$-theory group
$K^{\TT}_0(C^*(E))$ is order isomorphic, as a $\ZZ[x, x^{-1}]$-module, to the $K_0$-group
$K_0(C^*(E) \times_\gamma \TT)$. The canonical isomorphism $C^*(E) \times_\gamma \TT \cong C^*(E
\times_1 \ZZ)$ is equivariant for the dual action $\hat\gamma$ of $\ZZ$ on the former and the
action of $\ZZ$ on the latter induced by translation in $E \times_1 \ZZ$. It therefore induces an
isomorphism $K_0(C^*(E) \times_\gamma \TT) \cong K_0(C^*(E \times_1 \ZZ))$  of ordered $\ZZ[x,
x^{-1}]$-modules. So condition~(\ref{it:main C* K-th}) holds if and only if $K_0(C^*(E \times_1
\ZZ)) \cong K_0(C^*(F \times_1 \ZZ))$ as ordered $\ZZ[x, x^{-1}]$-modules.

By \cite[Theorem~3.4 and Corollary~3.5]{HLMRT} (see also \cite{AG}), for any directed graph $E$
there is an isomorphism $K_0(L_{\KK}(E)) \cong K_0(C^*(E))$ that carries the class of the module
$L_{\KK}(E) v$ to the class of the projection $p_v$ in $C^*(E)$ for each $v \in E^0$. It follows
that $K_0(L_{\KK}(E \times_1 \ZZ)) \cong K_0(C^*(E \times_1 \ZZ))$ as ordered $\ZZ[x,
x^{-1}]$-modules. This shows that conditions (\ref{it:main alg K-th})~and~(\ref{it:main C* K-th})
are equivalent. So it now suffices to show that (\ref{it:main alg K-th})~implies~(\ref{it:main
graph iso}).

So suppose that~(\ref{it:main alg K-th}) holds. Since $E$, and therefore $E \times_1 \ZZ$, is an
amplified graph, it admits no breaking vertices with respect to any saturated hereditary set, and
every hereditary subset of $E \times_1 \ZZ$ is a saturated hereditary subset. So the lattice
$\Hh(E \times_1 \ZZ)$ of hereditary sets is identical to the lattice of admissible pairs in the
sense of \cite{Tomforde} via the map $H \mapsto (H, \emptyset)$. By \cite[Theorem~5.11]{AHLS},
there is a lattice isomorphism from $\Hh(E \times_1 \ZZ)$ to the lattice of order ideals of
$K_0(L_{\KK}(E \times_1 \ZZ))$ that carries a hereditary set $H$ to the class of the module
$L_{\KK}(E \times_1 \ZZ)H$. this isomorphism clearly intertwines the action of $\ZZ$ induced by
the module structure on $K_0(L_{\KK}(E \times_1 \ZZ))$ and the action $\lt^E$ of $\ZZ$ on $\Hh(E
\times_1 \ZZ)$ induced by translation. By the same argument applied to $F$, we see that
$\big(\Hh(E \times_1 \ZZ), \subseteq, \lt^E) \cong \big(\Hh(F \times_1 \ZZ), \subseteq,
\lt^F\big)$.

Now Corollary~\ref{cor:H carried} implies that $\big(\Hh(E \times_1 \ZZ), \lt^E, H^E_0) \cong
\big(\Hh(F \times_1 \ZZ), \lt^F, H^F_0\big)$. This isomorphism induces an isomorphism
$\overline{E} \cong \overline{F}$ of the graphs constructed from these data in
Proposition~\ref{prp:key}. Thus two applications of Proposition~\ref{prp:key} give $E \cong
\overline{E} \cong \overline{F} \cong F$, which is~(\ref{it:main graph iso}).
\end{proof}

\section{Equivariant \texorpdfstring{$K$}{K}-theory and graded \texorpdfstring{$K$}{K}-theory are stable invariants}

In this section, we prove that equivariant $K$-theory and graded $K$-theory are stable invariants.
We suspect that these are well-known results but we have been unable to find a reference in the
literature.  For the convenience of the reader, we include their proofs here. We use these results
to deduce the consequences of Theorem~\ref{thm-main} for graded stable isomorphisms of amplified
Leavitt path algebras, and gauge-equivariant stable isomorphisms of amplified graph
$C^*$-algebras.

\subsection{Stability of equivariant \texorpdfstring{$K$}{K}-theory}

\begin{thm}\label{t:julgs-theorem}
Let $G$ be a compact group and let $\alpha$ be an action of $G$ on a $C^*$-algebra $A$.  If $A$
has an increasing approximate identity consisting of $G$-invariant projections, then the natural
$R(G)$-module isomorphism from $K_0^G ( A, \alpha )$ to $K_0 ( C^*(G, A , \alpha ))$ is an order
isomorphism.
\end{thm}

\begin{proof}
First suppose $A$ has a unit.  Then the theorem follows from the proof of Julg's Theorem
\cite{Julg} (see also \cite[Theorem~2.7.9]{NCP-LectureNotes}).  The isomorphism is given by the
composition of two isomorphisms:
\begin{align*}
K_0^G( A, \alpha  ) &\to K_0 ( L^1 ( G, A , \alpha) ) \quad \text{and}\\
K_0 ( L^1 ( G, A , \alpha) ) &\to K_0 ( C^*( G, A, \alpha ) ).
\end{align*}
The proof that these maps are isomorphisms shows that the maps are order isomorphisms (see the
proof of \cite[Lemma~2.4.2 and Theorem~2.6.1]{NCP-LectureNotes}).

Now suppose $A$ has an increasing approximate identity $S$ consisting of $G$-invariant
projections. Fix $p \in S$. Let
\begin{align*}
\lambda_A &\colon K_0^G(A , \alpha) \to K_0 ( C^*(G, A, \alpha ) ),\text{ and}\\
\lambda_p &\colon K_0 ^G( p A p, \alpha ) \to K_0 ( C^*( G, p A p , \alpha ) ),\quad p \in S
\end{align*}
be the natural $R(G)$-isomorphisms given in Julg's Theorem.  Note that $\alpha$ does indeed induce
an action on $pAp$ since $p$ is $G$-invariant.  Let $\iota_p$ be the $G$-equivariant inclusion of
$pAp$ into $A$ and let $\widetilde{\iota_p}$ be the induce $*$-homomorphism from $C^*(G, p A p,
\alpha )$ to $C^*(G, A , \alpha )$.

Let $x \in K_0^G(A, \alpha)_+$.  By \cite[Corollary~2.5.5]{NCP-LectureNotes}, there exist $p \in
S$ and $x' \in K_0^G( p A p, \alpha )_+$ such that $(\iota_p)_*( x' ) = x$. Naturality of the maps
$\lambda_A$ and $\lambda_p$ gives $\lambda_A (x) = (\widetilde{\iota}_p)_* \circ \lambda_p ( x')$.
Consequently, $\lambda_A (x) \in K_0 ( C^*(G, A, \alpha ) )_+$ since $(\widetilde{\iota}_p)_*
\circ \lambda_p ( x') \in K_0 ( C^*(G, A, \alpha ) )_+$.  Fix $y \in K_0 ( C^*(G, A, \alpha )
)_+$. For $f \in L^1(G)$ and $a \in A$ we write $f \otimes a : G \to A$ for the function
$(f\otimes a)(g) = f(g)a$. Since $S$ is an approximate identity of $A$ and since
\[
\{ f \otimes a : f \in L^1(G), a \in A \}
\]
is dense in $C^*(G, A, \alpha )$, the set $\bigcup_{ p \in S } \widetilde{ \iota_p } ( C^*( G, p A
p , \alpha ) )$ is dense in $C^*(G, A, \alpha )$. Thus, there exists a projection $p \in S$ and
there exists $y' \in K_0 ( C^*(G, p A p , \alpha ) )_+$ such that $(\widetilde{\iota_p})_*(y' ) =
x$. Since $\lambda_p$ is an order isomorphism, $\lambda_p^{-1} ( y' ) \in K_0^G( p A p , \alpha
)_+$. Then $( \iota_p )_* \circ \lambda_p^{-1} ( y' ) \in K_0^G( A, \alpha )_+$.  Naturality of
the maps $\lambda_A$ and $\lambda_p$ implies that $\lambda_A \circ ( \iota_p )_* \circ
\lambda_p^{-1} ( y' ) = y$.  We have shown that $\lambda_A ( K_0^G(A , \alpha )_+ ) = K_0 ( C^*(G,
A, \alpha ))_+$ which implies that $\lambda_A$ is an order isomorphism.
\end{proof}

\begin{lem}\label{l:fullher}
Let $G$ be a compact group and let $A$ be a separable $C^*$-algebra and let $\alpha$ be an action
of $G$ on $A$.  If $B$ is a hereditary subalgebra of $A$ such that
\begin{enumerate}
\item $B$ has an increasing approximate identity of $G$-invariant projections,

\item $A$ has an increasing approximate identity of $G$-invariant projections,

\item $\overline{A B A } =A$, and

\item $\alpha_g ( B ) \subseteq B$ for all $g \in G$,
\end{enumerate}
then the inclusion $\iota\colon B \to A$ induces an isomorphism $K_0^G (B) \cong K_0^G(A)$ of
ordered $R(G)$-modules.
\end{lem}

\begin{proof}
Since $B$ is $G$-invariant, $\alpha$ is also an action on $B$ and the inclusion $\iota$ is
$G$-equivariant.  Let $\lambda_B \colon K^G_0 ( B , \alpha ) \to K_0 ( C^*( G, B , \alpha ) )$ and
$\lambda_A \colon K_0^G ( A ) \to K_0 ( C^*( G, A , \alpha ) )$ be the natural $R(G)$-module order
isomorphisms given in Theorem~\ref{t:julgs-theorem}.  Naturality of $\lambda_B$ and $\lambda_A$
implies that the diagram
\[
\xymatrix{
K_0^G ( B ) \ar[r]^{ \iota_* }\ar[d]_{\lambda_B} & K_0^G(A) \ar[d]^{\lambda_A} \\
K_0 ( C^*( G, B, \alpha ) ) \ar[r]_{ \widetilde{\iota}_* } & K_0 ( C^* ( G, A , \alpha ) )
}
\]
is commutative. As in the proof of \cite[Proposition~2.9.1]{NCP-LectureNotes}, $C^*( G, B, \alpha
)$ is a hereditary subalgebra of $C^*( G, A, \alpha )$ such that the closed two-sided ideal of
$C^*(G, A, \alpha )$ generated by $C^*( G, B, \alpha )$ is $C^*(G, A,\alpha)$. This
$\widetilde{\iota}_*$ is an order isomorphism, and so $\iota_*$ is also an order isomorphism.
\end{proof}

The corollary below implies that the equivariant $K_0$-group is a stable invariant.

\begin{cor}\label{cor-stabilizeeqkthy}
Let $G$ be a compact group, let $\alpha$ be an action of $G$ on a separable $C^*$-algebra $A$, and
let $\beta$ be an action of $G$ on $\Kk(\ell^2)$.  If both $A$ and $\Kk(\ell^2)$ admit increasing
approximate identities consisting of $G$-invariant projections, then there is a $R(G)$-module
order isomorphism from $K_0^G (A, \alpha )$ to $K_0^G( A \otimes \Kk(\ell^2) , \alpha \otimes
\beta)$.

In particular, if $u : G \to \Uu(\ell^2)$ is a continuous (in the strong operator topology)
unitary representation of $G$ and $\beta_g = \mathrm{Ad}(u_g)$, then there is a $R(G)$-module
order isomorphism from $K_0^G (A, \alpha )$ and $K_0^G( A \otimes \Kk(\ell^2) , \alpha \otimes
\beta)$
\end{cor}

\begin{proof}
Let $\{ p_n \}_{ n \in \NN}$ be an increasing approximate identity consisting of $G$-invariant
projections in $\Kk(\ell^2)$.  We may assume $p_1 \neq 0$.  Then $A \otimes p_1$ is a
$G$-invariant hereditary subalgebra of $A \otimes \Kk(\ell^2)$ such that $\overline{( A \otimes
\Kk(\ell^2)) (A \otimes p_1 )( A \otimes \Kk(\ell^2) ) } = A \otimes \Kk(\ell^2)$.  From the
assumption on $A$ and $\Kk(\ell^2)$, both $A \otimes p_1$ and $A \otimes \Kk(\ell^2)$ have
increasing approximate identities consisting of $G$-invariant projections.  Lemma~\ref{l:fullher}
implies that there is an $R(G)$-module order isomorphism from $K_0^G(A \otimes p_1, \alpha \otimes
\beta )$ to $K_0^G (A \otimes \Kk(\ell^2) , \alpha \otimes \beta)$.  The result now follows since
the map $a \mapsto a \otimes p_1$ is a $G$-equivariant $*$-isomorphism from $A$ to $A\otimes p_1$.

For the last part of the theorem, since $G$ is compact, $u$ is a direct sum of finite dimensional
representations.  Thus, $\Kk(\ell^2)$ has an increasing approximate identity consisting of
$G$-invariant projections.
\end{proof}

To finish this subsection, we describe the consequences of Theorem~\ref{thm-main} for equivariant
stable isomorphism of amplified graph $C^*$-algebras. For the following theorem, given a
strong-operator continuous unitary representation $u : \TT \to \Uu(\ell^2)$ of $\TT$ on a Hilbert
space $H$, we will write $\beta^u$ for the action of $\TT$ on $\Bb(\ell^2)$ given by $\beta^u_z =
\mathrm{Ad}(u_z)$.

\begin{thm}\label{thm:C*-consequences}
Let $E$ and $F$ be countable amplified graphs. Then the following are equivalent:
\begin{enumerate}
\item\label{it:C* graphs iso} $E \cong F$;
\item\label{it:C* iso} $(C^*(E), \gamma^E) \cong (C^*(F), \gamma^F)$;
\item\label{it:C* stable any u} $(C^*(E) \otimes \Kk, \gamma^E \otimes \beta^u) \cong (C^*(F)
    \otimes \Kk, \gamma^F \otimes \beta^u)$, for every strongly continuous representation $u :
    \TT \to \Uu(\ell^2)$;
\item\label{it:C* stable some u} there exists a strongly continuous unitary representation $u :
    \TT \to \Uu(\ell^2)$ such that $(C^*(E) \otimes \Kk, \gamma^E \otimes \beta^u) \cong (C^*(F)
    \otimes \Kk, \gamma^F \otimes \beta^u)$; and
\item\label{it:C* stable some u,v} there exist strongly continuous unitary representations $u,v
    : \TT \to \Uu(\ell^2)$ such that $(C^*(E) \otimes \Kk, \gamma^E \otimes \beta^u) \cong
    (C^*(F) \otimes \Kk, \gamma^F \otimes \beta^v)$.
\end{enumerate}
\end{thm}
\begin{proof}
If $\phi : E \to F$ is an isomorphism, it induces an isomorphism $C^*(E) \cong C^*(F)$, which is
gauge invariant because it carries generators to generators. This gives \mbox{(\ref{it:C* graphs
iso})\;$\implies$\;(\ref{it:C* iso})}.

If~(\ref{it:C* iso}) holds, say $\phi : C^*(E) \to C^*(F)$ is a gauge-equivariant isomorphism,
then for any $u$ the map $\phi \otimes \id_\Kk$ is an equivariant isomorphism from $(C^*(E)
\otimes \Kk, \gamma^E \otimes \beta^u)$ to $(C^*(F) \otimes \Kk, \gamma^F \otimes \beta^u)$,
giving~(\ref{it:C* stable any u}). Clearly~(\ref{it:C* stable any u}) implies~(\ref{it:C* stable
some u}). And if~(\ref{it:C* stable some u}) holds for a given $u : \TT \to \Bb(\ell^2)$,
then~(\ref{it:C* stable some u,v}) holds with $u = v$. Finally, if~(\ref{it:C* stable some u,v})
holds, then two applications of Corollary~\ref{cor-stabilizeeqkthy} show that
\begin{align*}
K_0^\TT(C^*(E), \gamma^E) &\cong K_0^\TT( C^*(E) \otimes \Kk(\ell^2) , \gamma^E \otimes \beta^u)\\
    &\cong K_0^\TT( C^*(F) \otimes \Kk(\ell^2) , \gamma^F \otimes \beta^v)
    \cong K_0^\TT(C^*(F), \gamma^F)
\end{align*}
as ordered $\ZZ[x, x^{-1}]$-modules, and so Theorem~\ref{thm-main} gives~(\ref{it:C* graphs iso}).
\end{proof}

\begin{rmk}\label{rm:ER conjecture}
In this remark, we use the notation, moves, and drawing conventions of \cite{ER}; we refer the
reader there for details.

Combined with the results of others, Theorem~\ref{thm:C*-consequences} confirms, for the class of
amplified graphs, \cite[Conjecture~5.1]{ER}. The conjecture states that for all $\mathsf{xyz}$
other than $\mathsf{010}$ and $\mathsf{101}$, the equivalence relation
$\overline{\underline{\mathsf{x\smash{\mathsf{y}}z}}}$ is generated by the moves from
$\{\texttt{(O)}, \texttt{(I-)}, \texttt{(I+)}, \texttt{(R+)}, \texttt{(S)}, \texttt{(C+)},
\texttt{(P+)}\}$ that preserve it.

Theorem~\ref{thm:C*-consequences} shows that for amplified graphs,
\[
(E, F) \in \overline{\underline{\mathsf{010}}} \implies E \cong F \implies (E, F) \in \overline{\underline{\mathsf{111}}}
\]
Since we trivially have $\overline{\underline{\mathsf{111}}} \subseteq
\overline{\underline{\mathsf{x1z}}} \subseteq \overline{\underline{\mathsf{010}}}$ for all
$\mathsf{x,z}$, we deduce that the four equivalence relations
$\overline{\underline{\mathsf{x1z}}}$ are identical and coincide with graph isomorphism. In
particular, for amplified graphs, each $\overline{\underline{\mathsf{x1z}}}$ is trivially
contained in the relation generated by the moves that preserve it. For the reverse containment,
note that the only moves in the list above that preserve any $\mathsf{x1z}$-equivalences are
\texttt{(O), (I+)} and \texttt{(I-)}. Of these, neither \texttt{(I+)} nor \texttt{(I-)} can be
applied to an amplified graph, and \cite[Theorem~3.2]{ER} shows that $\langle\texttt{(O)}\rangle
\subseteq \overline{\underline{\mathsf{x1z}}}$ for all $x,z$. So we confirm
\cite[Conjecture~5.1]{ER} for amplified graphs for the relations
$\overline{\underline{\mathsf{x1z}}}$.

We now show that a similar result holds for the relations $\overline{\underline{\mathsf{x0z}}}$.
Recall from \cite{ERS-amplified} that if $E$ is an amplified graph then its amplified transitive
closure $tE$ is the amplified graph with $tE^0 = E^0$ and $v(tE^1)w \not= \emptyset$ if and only
if $vE^*w \setminus \{v\}\not= \emptyset$. Theorem~1.1 of \cite{ERS-amplified} shows that for
amplified graphs, if $(E, F) \in \overline{\underline{\mathsf{000}}}$, then $tE \cong tF$. We
claim that this forces $(E, F) \in \overline{\underline{\mathsf{101}}}$. To see this, first note
that by \cite[Theorems 3.2~and~3.10]{ER}, moves \texttt{(0)} and \texttt{(R+)} preserve
$\overline{\underline{\mathsf{101}}}$. So it suffices to show that the graph move $t$ that, given
vertices $u,v,w$ such that $u E^1 v$ and $vE^1 w$ are infinite, adds infinitely many new edges to
$u E^1 w$, can be obtained using \texttt{(0)} and \texttt{(R+)}. This is achieved as follows:
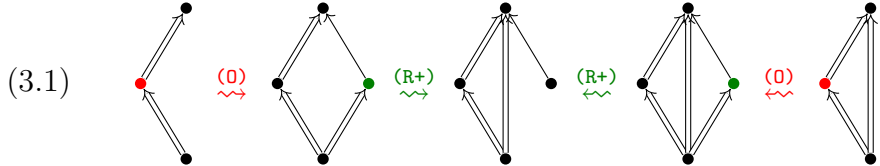
\begin{equation}\label{eq:ORRO}
\tikzset{infedge/.style={
    arrows={-Implies},
    double,
    double distance = 1.5pt
}}
\parbox[c]{0.8\textwidth}{
\begin{tikzpicture}[xscale=1.2]
    \node[circle, inner sep=1.5pt, fill=black] (u0) at (0.5,0) {};
    \node[circle, inner sep=1.5pt, fill=red] (v0) at (0,1) {};
    \node[circle, inner sep=1.5pt, fill=black] (w0) at (0.5,2) {};
    \draw[infedge] (u0)--(v0);
    \draw[infedge] (v0)--(w0);
    \node at (1,1) {\color{red}$\stackrel{\texttt{(O)}}{\rightsquigarrow}$};
    \node[circle, inner sep=1.5pt, fill=black] (u1) at (2,0) {};
    \node[circle, inner sep=1.5pt, fill=black] (v1l) at (1.5,1) {};
    \node[circle, inner sep=1.5pt, fill=green!50!black] (v1r) at (2.5,1) {};
    \node[circle, inner sep=1.5pt, fill=black] (w1) at (2,2) {};
    \draw[infedge] (u1)--(v1l);
    \draw[infedge] (u1)--(v1r);
    \draw[infedge] (v1l)--(w1);
    \draw[->] (v1r)--(w1);
    \node at (3,1) {\color{green!50!black}$\stackrel{\texttt{(R+)}}{\rightsquigarrow}$};
    \node[circle, inner sep=1.5pt, fill=black] (u2) at (4,0) {};
    \node[circle, inner sep=1.5pt, fill=black] (v2l) at (3.5,1) {};
    \node[circle, inner sep=1.5pt, fill=black] (v2r) at (4.5,1) {};
    \node[circle, inner sep=1.5pt, fill=black] (w2) at (4,2) {};
    \draw[infedge] (u2)--(v2l);
    \draw[infedge] (u2)--(w2);
    \draw[infedge] (v2l)--(w2);
    \draw[->] (v2r)--(w2);
    \node at (5,1) {\color{green!50!black}$\stackrel{\texttt{(R+)}}{\leftsquigarrow}$};
    \node[circle, inner sep=1.5pt, fill=black] (u3) at (6,0) {};
    \node[circle, inner sep=1.5pt, fill=black] (v3l) at (5.5,1) {};
    \node[circle, inner sep=1.5pt, fill=green!50!black] (v3r) at (6.5,1) {};
    \node[circle, inner sep=1.5pt, fill=black] (w3) at (6,2) {};
    \draw[infedge] (u3)--(v3l);
    \draw[infedge] (u3)--(w3);
    \draw[infedge] (u3)--(v3r);
    \draw[infedge] (v3l)--(w3);
    \draw[->] (v3r)--(w3);
    \node at (7,1) {\color{red}$\stackrel{\texttt{(O)}}{\leftsquigarrow}$};
    \node[circle, inner sep=1.5pt, fill=black] (u4) at (8,0) {};
    \node[circle, inner sep=1.5pt, fill=red] (v4) at (7.5,1) {};
    \node[circle, inner sep=1.5pt, fill=black] (w4) at (8,2) {};
    \draw[infedge] (u4)--(v4);
    \draw[infedge] (v4)--(w4);
    \draw[infedge] (u4)--(w4);
\end{tikzpicture}}
\end{equation}
So as above, for amplified graphs, we see that the four equivalence relations
$\overline{\underline{\mathsf{x0z}}}$ are identical, coincide with isomorphism of amplified
transitive closures of the underlying graphs, and are generated by \texttt{(O)} and \texttt{(R+)},
and in particular by the moves from \cite{ER} that are $\mathsf{x0z}$-invariant. The results of
\cite{ER} give the reverse containment, so we have confirmed \cite[Conjecture~5.1]{ER} for
amplified graphs for the relations $\overline{\underline{\mathsf{x0z}}}$.
\end{rmk}

\subsection{Stability of graded algebraic \texorpdfstring{$K_0$}{K0}}\label{sec:stabilizations}

Next we establish the stable invariance of graded $K$-theory.  Let $\Gamma$ be an additive abelian
group and let $A$ be a $\Gamma$-graded ring.  For $\overline{\delta} \in \Gamma^n$, we write $M_n
(A) ( \overline{\delta})$ for the $\Gamma$-graded ring $M_n (A)$ with grading given by $( a_{i, j}
) \in M_n (A)_\lambda$ if and only if $a_{i, j} \in A_{\lambda + \delta_j - \delta_i}$. Similarly,
for $\overline{\delta} \in \prod_n \Gamma$, we write $M_\infty (A) ( \overline{\delta})$ for the
$\Gamma$-graded ring $M_\infty (A)$ with grading given by $( a_{i, j} ) \in M_\infty (A) (
\overline{\delta})_\lambda$ if and only if $a_{i, j} \in  A_{\lambda + \delta_j - \delta_i}$.

Since the tensor product of two graded modules will be key in the proof, we recall the construct
given in \cite[Section~1.2.6]{RH-graded}.  Let $\Gamma$ be an additive abelian group, let $A$ be a
$\Gamma$-graded ring, let $M$ be a graded right $A$-module, and let $N$ be a graded left
$A$-module.  Then $M \otimes_A N$ is defined to be $M \otimes_{A_0} N$ modulo the subgroup
generated by
\[
\{ m a \otimes n - m \otimes an : \text{$m \in M, n \in N,$ and $a \in A$ are homogeneous} \}
\]
with grading induced by the grading on $M \otimes_{A_0} N$ given by
\[
(M \otimes_{ A_0 } N )_\lambda = \left\{ \sum_{ i } m_i \otimes n_i : \text{$m_i \in M_{\alpha_i}, n_i \in N_{\beta_i}$ with $\alpha_i + \beta_i = \lambda$ } \right\}.
\]

\begin{thm}\label{thm:full-corner}
Let $\Gamma$ be an additive abelian group, let $A$ be a unital $\Gamma$-graded ring, and let
$\overline{\delta} = ( \delta_1, \delta_2, \ldots, \delta_n ) \in \Gamma^n$.  Then the inclusion
$\iota \colon A \to M_n ( A ) ( \overline{\delta} )$ into the $e_{1,1}$ corner induces a $\ZZ[
\Gamma ]$-module order isomorphism $K_0^{\gr}( \iota ) \colon K_0^{\gr} (A) \to K_0^{\gr}( M_n ( A
) ( \overline{\delta} ) )$ given by $K_0^{\gr}( \iota ) ([P]) = [ P \otimes_{ A } M_n ( A ) (
\overline{\delta} ) ]$ (the left $A$-module structure on $M_n ( A ) ( \overline{\delta} )$ is
given by the inclusion $\iota$).
\end{thm}

\begin{proof}
Let $\overline{\alpha} = (0, \delta_2 - \delta_1, \ldots, \delta_n - \delta_1)$.  By
\cite[Corollary~2.1.2]{RH-graded}, there is an equivalence of categories $\phi \colon
\mathrm{Pgr}\hyphen A \to \mathrm{Pgr}\hyphen M_n ( A ) ( \overline{\alpha} )$ given by $\phi ( P
) = P \otimes_A A^n ( \overline{\alpha} )$.   Moreover, $\phi$ commutes with the suspension map.
Since
\begin{align*}
M_n (A) ( \overline{\alpha}  )_\lambda &= \begin{pmatrix}  A_{ \lambda  } & A_{ \lambda + \alpha_2 - \alpha_1} & \cdots & A_{ \lambda + \alpha_n - \alpha_1 } \\
A_{ \lambda + \alpha_1 - \alpha_2  } & A_{ \lambda } & \cdots & A_{ \lambda + \alpha_n - \alpha_2 } \\
\vdots & \vdots & \ddots & \vdots \\
A_{ \lambda + \alpha_1 - \alpha_n  } & A_{ \lambda + \alpha_2 - \alpha_n } & \cdots & A_{ \lambda  }
\end{pmatrix} \\
&=
 \begin{pmatrix}  A_{ \lambda  } & A_{ \lambda + \delta_2 - \delta_1} & \cdots & A_{ \lambda + \delta_n - \delta_1 } \\
A_{ \lambda + \delta_1 - \delta_2  } & A_{ \lambda } & \cdots & A_{ \lambda + \delta_n - \delta_2 } \\
\vdots & \vdots & \ddots & \vdots \\
A_{ \lambda + \delta_1 - \delta_n  } & A_{ \lambda + \delta_2 - \delta_n } & \cdots & A_{ \lambda  }
\end{pmatrix}
= M_n (A) ( \overline{\delta}  )_\lambda,
\end{align*}
we have $M_n (A)( \overline{\alpha} ) = M_n ( A) ( \overline{\delta} )$.  Therefore, $\phi ( P ) =
P \otimes_A A^n ( \overline{\alpha} )$ is an equivalence of categories from $\mathrm{Pgr}\hyphen
A$ to $\mathrm{Pgr}\hyphen M_n ( A ) ( \overline{\delta} )$ and $\phi$ commutes with the
suspension map.  Hence, $\phi$ induces a $\ZZ[ \Gamma ]$-module order isomorphism from $K_0^{\gr}
( A )$ to $K_0^{\gr} ( M_n ( A ) ( \overline{\delta} ) )$.

We claim that $\phi = K_0^{\gr} ( \iota )$.  Let $M$ be a graded right $A$-module.  We will show
that $M \otimes_A A^n ( \overline{\alpha} )$ and $M \otimes_{ A } M_n ( A ) ( \overline{\delta} )$
are isomorphic as graded modules.  Since $1_A \in A_0$ and $M1_A = M$,
\[
M \otimes_{ A } M_n ( A ) ( \overline{\delta} ) \cong_{\gr} M \otimes_{ A } \iota ( 1_A)  M_n ( A ) ( \overline{\delta} ) = M \otimes_{ A } e_{1,1} M_n ( A ) ( \overline{\delta} ).
\]
By the definitions of the gradings on $e_{1,1} M_n ( A ) ( \overline{\delta} )$ and $A^n ( \alpha
)$, the right $M_n (A)$-module isomorphism
\[
e_{1,1} X \mapsto ( x_{1,1} , x_{1, 2} , \ldots, x_{1, n } )
\]
is a graded isomorphism.  Hence,
\[
M \otimes_{ A } M_n ( A ) ( \overline{\delta} ) \cong_{\gr} M \otimes_{ A } e_{1,1} M_n ( A ) ( \overline{\delta} ) \cong_{\gr} M \otimes_{A} A^n ( \alpha ).
\]
Thus, $\phi = K_0^{\gr} ( \iota )$.  Consequently, $K_0^{\gr} ( \iota )$ is a $\ZZ [ \Gamma
]$-module order isomorphism.
\end{proof}

\begin{cor}\label{cor:stabilize-gradedkthy}
Let $\Gamma$ be an additive abelian group and let $A$ be a $\Gamma$-graded ring with a sequence of
idempotents $\{e_n\}_{n=1}^\infty \subseteq A_0$ such that $e_n e_{n+1} = e_n$ for all $n$, and
$\bigcup_{n} e_n Ae_n = A$. For $\overline{\delta} \in \prod_i \Gamma$, the embedding $\iota
\colon A \to M_\infty (A)( \overline{\delta})$ into the $e_{1,1}$ corner of $M_\infty ( A ) (
\overline{\delta} )$ induces a $\ZZ [ \Gamma ]$-module order isomorphism $K_0^{\gr} (\iota )
\colon K^{\gr} ( A ) \to K_0^{\gr} ( M_\infty (A)( \overline{\delta}) )$.

In particular, if $E$ is a countable directed graph and $\overline{\delta} \in \prod_i \ZZ$, then
the inclusion of $\iota \colon L_\KK (E) \to M_\infty ( L_\KK (E))( \overline{\delta})$ of $L_\KK
(E)$ into the $e_{1,1}$ corner of $M_\infty ( L_\KK (E))( \overline{\delta})$ induces a $\ZZ[x,
x^{-1} ]$-module order isomorphism from $K_0^{\gr} ( L_\KK (E) )$ to $K_0^{\gr}(  M_\infty ( L_\KK
(E))( \overline{\delta}) )$ for any field $\KK$.
\end{cor}

\begin{proof}
Let $\iota_n \colon e_n A e_n \to M_{\infty} ( e_n A e_n )(\overline{\delta})$ be the inclusion of
$e_n Ae_n$ into the $e_{1,1}$ corner of $M_{\infty} ( e_n A e_n )(\overline{\delta})$.  Observe
that $A = \varinjlim e_n A e_n$, that $M_\infty ( A ) = \varinjlim M_\infty ( e_n A e_n )$, and
that the diagram
\[
\xymatrix{
e_n A e_n \ar[r]^{\subseteq} \ar[d]_{ \iota_n} & A \ar[d]^{\iota} \\
M_\infty ( e_n A e_n ) (\overline{\delta}) \ar[r]^{\subseteq}  & M_\infty ( A  ) ( \overline{\delta})
}
\]
commutes. Therefore, if each $K_0^{\gr}(\iota_n)$ is a $\ZZ [ \Gamma ]$-module order isomorphism,
then $K_0^{\gr} (\iota)$ is a $\ZZ [ \Gamma ]$-module order isomorphism since the graded
$K_0$-group respects direct limits (\cite[Theorem~3.2.4]{RH-graded}).  Hence, without loss of
generality, we may assume that $A$ is a unital $\Gamma$-graded ring.

Let $\overline{\delta}_n = ( \delta_1, \delta_2, \ldots, \delta_n )$. Let $j_n \colon A \to M_n (
A )( \overline{\delta}_n )$ be the inclusion of $A$ into the $e_{1,1}$ corner of $M_n ( A )(
\overline{\delta}_n )$, and let $\iota_n : M_n(A)(\overline{\delta}_n) \to
M_\infty(A)(\overline{\delta})$ be the inclusion map. Then $\varinjlim M_n ( A ) (
\overline{\delta}_n )  = M_\infty ( A ) ( \overline{\delta})$ and the diagram
\[
\xymatrix{
A  \ar[d]_{ j_n } \ar[rd]^{ \iota }  & \\
M_n ( A ) (\overline{\delta}_n ) \ar[r]^{\iota_n} & M_{\infty} (A) ( \overline{\delta} )
}
\]
commutes. By Theorem~\ref{thm:full-corner}, $K_0^{\gr} ( j_n )$ is a $\ZZ[ \Gamma ]$-module order
isomorphism. Since the graded-$K_0$ functor respects direct limits, $K_0^{\gr}(\iota)$ is $\ZZ[
\Gamma ]$-module order isomorphism.

For the last part of the corollary, let $\{ X_n \}$ be a sequence of finite subsets of $E^0$ such
that $X_n \subseteq X_{n+1}$ and $\bigcup_n X_n = E^0$.  Then $e_n := \sum_{ v \in X_n } v$
defines idempotents of degree zero such that $\bigcup_n e_n L_\KK ( E) e_n = L_\KK (E)$.
\end{proof}

As in the preceding subsection, we finish by describing the consequences of Theorem~\ref{thm-main}
for graded stable isomorphism of amplified Leavitt path algebras.

\begin{thm}\label{thm:LPA-consequences}
Let $E$ and $F$ be countable amplified graphs and let $\KK$ be a field. Then the following are
equivalent:
\begin{enumerate}
\item\label{it:LPA graphs iso} $E \cong F$
\item\label{it:LPA iso} $L_\KK(E) \cong^{\gr} L_\KK(F)$;
\item\label{it:LPA stable any delta} $L_\KK(E) \otimes M_\infty(\KK)(\overline{\delta})
    \cong^{\gr} L_\KK(F) \otimes M_\infty(\KK)(\overline{\delta})$ for every $\overline{\delta}
    \in \prod_i \ZZ$;
\item\label{it:LPA stable some delta} $L_\KK(E) \otimes M_\infty(\KK)(\overline{\delta})
    \cong^{\gr} L_\KK(F) \otimes M_\infty(\KK)(\overline{\delta})$ for some $\overline{\delta}
    \in \prod_i \ZZ$; and
\item\label{it:LPA stable some delta,eps} $L_\KK(E) \otimes M_\infty(\KK)(\overline{\delta})
    \cong^{\gr} L_\KK(F) \otimes M_\infty(\KK)(\overline{\varepsilon})$ for some
    $\overline{\delta}, \overline{\varepsilon} \in \prod_i \ZZ$.
\end{enumerate}
\end{thm}
\begin{proof}
The argument is very similar to that of Theorem~\ref{thm:C*-consequences}, so we summarise. Any
isomorphism of graphs induces a graded isomorphism of their Leavitt path algebras, and any graded
isomorphism $\phi : L_\KK(E) \cong L_\KK(F)$ amplifies to a graded isomorphism $\phi \otimes \id :
L_\KK(E) \otimes M_\infty(\KK)(\overline{\delta}) \cong L_\KK(F) \otimes
M_\infty(\KK)(\overline{\delta})$, giving \mbox{(\ref{it:LPA graphs
iso})\;$\implies$\;(\ref{it:LPA iso})\;$\implies$\;(\ref{it:LPA stable any delta})}. The
implications \mbox{(\ref{it:LPA stable any delta})\;$\implies$\;(\ref{it:LPA stable some
delta})\;$\implies$\;(\ref{it:LPA stable some delta,eps})} are trivial. The second statement of
Corollary~\ref{cor:stabilize-gradedkthy} shows that if~(\ref{it:LPA stable some delta,eps}) holds
then $K^{\gr}(L_\KK(E)) \cong K^{\gr}(L_\KK(F))$ as ordered $\ZZ[x, x^{-1}]$-modules, and then
Theorem~\ref{thm-main} gives~(\ref{it:LPA graphs iso}).
\end{proof}

\begin{rmk}\label{rmk:field indep}
Since statement~(\ref{it:LPA graphs iso}) of Theorem~\ref{thm:LPA-consequences} does not depend on
the field $\KK$, we deduce that each of the other four statements holds for some field $\KK$ if
and only if holds for every field $\KK$. In particular the graded-isomorphism problem for
amplified Leavitt path algebras is field independent, so it suffices, for example, to consider the
field $\mathbb{F}_2$.
\end{rmk}

\begin{rmk}\label{rmk:different stabilisations}
Let $E$ and $F$ be amplified graphs. Theorem~\ref{thm:C*-consequences} shows that the existence of
an isomorphism $(C^*(E) \otimes \Kk, \gamma^E \otimes \beta^u) \cong (C^*(F) \otimes \Kk, \gamma^F
\otimes \beta^u)$ for every $u$ is equivalent to the existence of such an isomorphism for some
$u$, and indeed to the existence of an isomorphism $(C^*(E) \otimes \Kk, \gamma^E \otimes \beta^u)
\cong (C^*(F) \otimes \Kk, \gamma^F \otimes \beta^v)$ for some $u,v$. All of these conditions are
formally weaker than the existence of isomoprhisms $(C^*(E) \otimes \Kk, \gamma^E \otimes \beta^u)
\cong (C^*(F) \otimes \Kk, \gamma^F \otimes \beta^v)$ for every pair of strongly continuous
representations $u,v : \TT \to \Uu(\ell^2)$, and this in turn is clearly equivalent to the
existence of an isomorphism $(C^*(E) \otimes \Kk, \gamma^E \otimes \beta^u) \cong (C^*(F) \otimes
\Kk, \gamma^F \otimes \id)$ for every $u$. So it is natural to ask for which amplified graphs $E,
F$ and which strongly continuous representations $u : \TT \to \Uu(\ell^2)$ we have $(C^*(E)
\otimes \Kk, \gamma^E \otimes \beta^u) \cong (C^*(F) \otimes \Kk, \gamma^F \otimes \id)$.

This is an intriguing question to which we do not know a complete answer, but we can certainly
show that the condition that $(C^*(E) \otimes \Kk, \gamma^E \otimes \beta^u) \cong (C^*(F) \otimes
\Kk, \gamma^F \otimes \id)$ for every $u$ is in general strictly stronger than the equivalent
conditions of Theorem~\ref{thm:C*-consequences}. Specifically, let $E = F$ be the directed graph
with $E^0 = \{v,w\}$ and $E^1 = \{e_n : n \in \NN\}$ with $s(e_n) = v$ and $r(e_n) = w$ for all
$\NN$. Then the only nonzero spectral subspaces for the gauge action on $C^*(E)$ are those
corresponding to $-1, 0, -1$, and so the same is true for the spectral subspaces of $C^*(E)
\otimes \Kk$ with respect to $\gamma^E \otimes \id$. On the other hand, if $u : \TT \to
B(\ell^2(\ZZ))$ is given by $u_z e_n = z^n e_n$, then each spectral subspace of $C^*(E) \otimes
\Kk$ for $\gamma^E \otimes \beta^u$ is nonempty, so $(C^*(E) \otimes \Kk, \gamma^E \otimes
\beta^u) \not\cong (C^*(E) \otimes \Kk, \gamma^E \otimes \id)$. We do not, however, know of an
example in which $C^*(E)$ is simple.

A similar question can be posed for amplified Leavitt path algebras: for which amplified graphs
$E, F$ and elements $\overline{\delta} \in \prod_i \ZZ$ do we have $L_\KK(E) \otimes
M_\infty(\KK)(\overline{\delta}) \cong^{\gr} L_\KK(F) \otimes M_\infty(\KK)(\overline{0})$? The
same example shows that the existence of such an isomorphism for every $\overline{\delta}$ is in
general strictly stronger than the equivalent conditions of Theorem~\ref{thm:LPA-consequences}.
\end{rmk}

\end{document}